\documentclass[11pt,reqno, oneside]{amsart}
\usepackage{hyperref}
\usepackage[dvipsnames]{xcolor}
\newcommand{\seqnum}[1]{\href{http://oeis.org/#1}{{#1}}}
\usepackage{enumerate}
\def\mytopsep{3mm}
\newtheoremstyle{myplain}{\mytopsep}{\mytopsep}{\itshape}{0pt}{\bfseries}{.}{3mm}{}
\newtheoremstyle{mydefinition}{\mytopsep}{\mytopsep}{\normalfont}{0pt}{\bfseries}{.}{3mm}{}
\newtheoremstyle{myremark}{\mytopsep}{\mytopsep}{\normalfont}{0pt}{\bfseries}{.}{3mm}{}
\theoremstyle{myplain}
\newtheorem{thm}{Theorem}[]

\newtheorem{lem}[thm]{Lemma}
\newtheorem{prop}[thm]{Proposition}

\numberwithin{equation}{section}

\newcommand{\inv}[1]{\left( 1#1\right)^{-1}}
\newcommand{\binv}[1]{\biggl(1#1\biggr)^{-1}}
\newcommand{\mathsum}{\mathop{\sum}\limits}
\newcommand{\N}{\mathbb N}
\newcommand{\floor}[1]{\left\lfloor#1\right\rfloor}
\newcommand{\w}{\omega}
\makeatletter
\usepackage{environ}
\NewEnviron{Lalign}{\tagsleft@true\begin{align}\BODY\end{align}}
\makeatother

\def\compspace{3mu} 
{\catcode`\,=13 \gdef,{\mskip \compspace}} 
\def\comp{\bgroup\catcode`\,=13\relax\compx}
\def\compx(#1){#1\egroup}

\begin{document}
\title{Compositions and Fibonacci Identities}
\author{Ira M. Gessel}
\address{Department of Mathematics, Brandeis University, MS 050, Waltham, MA 02453}
\email{gessel@brandeis.edu}
\author{Ji Li}
\email{vieplivee@gmail.com}
\thanks{This work was partially supported by a grant from the Simons Foundation (\#229238 to Ira Gessel).}
\date{\today}
\begin{abstract}
We study formulas expressing Fibonacci numbers as sums over compositions. For example, 
\begin{equation*}
F_{2n}=\sum a_1 a_2\cdots a_k
\end{equation*}
and for $n\ge 2$,
\begin{equation*}
F_{n-2}=\sum \floor{\frac{a_1-1}{2}}\cdots \floor{\frac{a_k-1}{2}}.
\end{equation*}
where the sums are over all compositions $\comp(a_1, a_2,\cdots, a_k)$ of $n$, for any $k$. We give a systematic account of such formulas using free monoids. The number of compositions of $n$ with parts 1 and 2 is the Fibonacci number $F_{n+1}$, and these compositions form a free monoid. Our formulas all come from free submonoids of this free monoid.
\end{abstract}
\maketitle
\smallskip

\section{Introduction}
\label{s-intro}
A composition of an integer $n$ is a sequence $\comp(a_1, a_2,\cdots, a_k)$ of positive integers, called the parts of the composition, with sum $n$.
Richard Stanley's \emph{Enumerative Combinatorics, Vol.~1\/} \cite[Chapter~1, Exercise 35, pp.~109 and 152--153]{ec1}, contains several formulas expressing Fibonacci numbers in terms of sums over compositions:
\begin{enumerate}[($i$)]
\item $F_{n+1}$ is the number of compositions of $n$ into parts equal to 1 or 2.\\[-10pt]
\item $F_{n-1}$ is the number of compositions of $n$ into  parts greater than 1.
\item $F_n$ is the number compositions of $n$ into odd parts.\\[-10pt]
\item $F_{2n} = \sum a_1 a_2\cdots a_k$.\\[-10pt]
\item $F_{2n-2} = \sum (2^{a_1 -1} -1)\cdots(2^{a_k -1}-1)$\\[-10pt]
\item $F_{2n+1} = \sum 2^{\#\{ i \,:\, a_i = 1\}}$
\end{enumerate}
Here the Fibonacci numbers are defined by $F_0=0$, $F_1=1$, and $F_n=F_{n-1}+F_{n-2}$ for $n\ge 2$, and the sums in $(iv)$--$(vi)$ are over all compositions $\comp(a_1,a_2,\cdots, a_k)$ of $n$. We note that these formulas are generally not true for $n=0$.

Our goal in this paper is to study identities of this form systematically, and to explain how to find such identities, how to prove them with generating functions, and how to prove them combinatorially.

It is easy to express sums over compositions in terms of generating functions. Let $C(n)$ be the set of compositions of $n$. Then we will write $\sum_{a\in C(n)}$ for a sum over all compositions $\comp(a_1,\cdots, a_k)$ of $n$, with any number of parts. The following result is well known and easily proved. (See, for example, Moser and Whitney \cite{mw} and  Hoggatt and Lind \cite{hl2}.)

\begin{lem}
\label{l-gf}
Let $u_1, u_2, u_3,\dots$ be any sequence of numbers. Then the sum
\[\sum_{a\in C(n)} u_{a_1}u_{a_2}\cdots u_{a_k}\]
is the coefficient of $x^n$ in
\[ \biggl(1-\sum_{i=1}^\infty u_ix^i\biggr)^{-1}.\]\end{lem}

Then formulas $(i)$--$(vi)$ follow from Lemma \ref{l-gf}, the easily verifiable identities
\begin{Lalign}
\frac 1{1-x-x^2} &= \inv {-x-x^2}\tag{$i'$}\\
1+\frac{x^2}{1-x-x^2} &= \inv{-x^2-x^3-x^4-\cdots}\tag{$ii'$}\\
 1+\frac x{1-x-x^2}&= \inv{-x-x^3-x^5-\cdots}\tag{$iii'$}\\
 1+\frac x{1-3x+x^2} &= \inv{-x-2x^2-3x^3-\cdots}\tag{$iv'$}\\
 1+\frac{x^2}{1-3x+x^2}&= \inv{-x^2-3x^3-7x^4-15x^5-\cdots}\tag{$v'$}\\
 \frac{1-x}{1-3x+x^2}&=\inv{-2x-x^2-x^3-x^4-x^5-\cdots},\tag{$vi'$}
 \end{Lalign}
and the formulas
\begin{subequations}
\begin{align}
\label{e-Fn}
\sum_{n=0}^\infty F_n x^n&= \frac{x}{1-x-x^2}\\\
\label{e-F2n}
\sum_{n=0}^\infty F_{2n}x^n &=\frac{x}{1-3x+x^2}\\
\label{e-F2n+1}
\sum_{n=0}^\infty F_{2n+1} x^n& = \frac{1-x}{1-3x+x^2}.
\end{align}
\end{subequations}
Formula \eqref{e-Fn} follows easily from the Fibonacci recurrence $F_n = F_{n-1}+F_{n-1}$. For \eqref{e-F2n} and \eqref{e-F2n+1}, we have
\begin{align}
\sum_{n=0}^\infty F_n x^n &= \frac{x}{1-x-x^2}\cdot \frac{1+x-x^2}{1+x-x^2}\notag\\
  &= \frac{x+x^2-x^3}{1-3x^2+x^4} = x\,\frac{1-x^2}{1-3x^2+x^4}
  +\frac{x^2}{1-3x^2+x^4}.\label{e-evenodd}
\end{align}
Then  \eqref{e-F2n} and \eqref{e-F2n+1} follow by extracting the even 
odd powers of $x$ from \eqref{e-evenodd}.

The proofs we have just sketched (essentially the generating function proofs given by Stanley), though straightforward, do not really explain why these formulas are true, nor how one might find them or other similar formulas. To do this, we study a combinatorial structure that lies behind them.

\section{Free Monoids}

Let $A$ be a set, which we call an \emph{alphabet}. Let $A^{\ast}$ be the set of words (finite sequences) of elements of $A$. Then with the operation of concatenation, $A$ is a monoid (a semigroup with unit), where the unit is the empty word. We call $A^{\ast}$ the \emph{free monoid on $A$}. The \emph{length} $l(x)$ of an element $x=a_1a_2\cdots a_k$, where each $a_i$ is in $A$, is $k$.  

More generally, a free monoid $M$ is a monoid isomorphic to a free monoid of the form $A^{\ast}$. So if $M$ is a free monoid, then there exists a subset $P$ of $M$ such that every element of $M$ has a unique factorization as a product of elements of $P$. We call $P$ the set of \emph{primes} of $M$. (It is easy to see that $P$ is unique.) 

A \emph{weight function} on a free monoid $M$ is a function $\w: M \to\N$, where $\N$ is the set of nonnegative integers, with the properties that $\w(m_1 m_2)=\w(m_1)+\w(m_2)$ for all $m_1, m_2\in M$ and $\w(m)=0$ if and only if $m$ is the unit element of $M$. It is easy to see that a weight function on $M$ is determined by its values on the primes of $M$. 

If $L$ is any submonoid of a free monoid, we call an element $p$ of $L$ \emph{irreducible} (in $L$) if $p$ is not the unit element of $M$ and $p$ cannot be expressed as a product of two non-unit elements of $L$. (Note that the irreducibility of $p$ depends on both $p$ and $L$.) It is clear that every element of $L$ can be factored as a product of irreducibles, but in general this factorization is not unique. If $L$ is a free monoid, then the factorization is always unique and the irreducible elements are the primes of $L$.  

Let $M$ be a free monoid with a weight function $\w$. If $x$ is an indeterminate then the map $m\mapsto x^{\w(m)}$ is a homomorphism from $M$ to the monoid of powers of $x$ under multiplication, and  unique factorization in $M$ gives  the well-known identity for formal power series
\begin{equation}
\label{e-freemonoid}
\sum_{m \in M} x^{\w(m)} = \biggl(1-\sum_{p \in P} x^{\w(p)}\biggr)^{-1},\end{equation}
where $P$ is the set of primes of $M$.
Equivalently, the number of words in $M$ of weight $n$ is the sum 
\[\sum_{a\in C(n)} u_{a_1}u_{a_2}\cdots u_{a_k},\] where $u_i$ is the number of primes of $M$ of weight $i$. We will see that formulas $(i)$--$(vi)$ can all be interpreted in this way.

We will be especially interested in free monoids that are submonoids of $A^{\ast}$ for some alphabet~$A$. For example, the set of words in $\{a\}^{\ast}$ of even length is a free submonoid of $\{a\}^*$. The set of words in $\{a\}^{\ast}$ of length not equal to 1 is also a submonoid but it is not free, since $a^5$ has two factorizations, $a^5= a^2\cdot a^3 = a^3\cdot a^2$ into words that cannot be further factored. The set of words in $\{a,b\}^*$ that start with $a$, together with the empty word, is a free submonoid in which the primes are of the form $ab^i$, for $i\in \N$.

Next we discuss some lemmas that are helpful in proving that submonoids of free monoids are free.

\begin{lem}
\label{l-prefix}
Suppose that $M$ is a submonoid of a free monoid $A^*$ with the property that every nonempty word in $M$ has a unique factorization of the form $xy$ where $x$ is irreducible in $M$ and $y\in M$. Then $M$ is free.
\end{lem}
\begin{proof}
We prove by induction on $n$ that every word in $M$ of length $n$ has a unique factorization into irreducibles of $M$. The assertion holds trivially for $n=0$. Now suppose that $w$ is a word in $M$ of length $n>0$ and that all words in $M$ of length less than $n$ have unique factorizations into irreducibles.  Let $w=x_1 x_2\cdots x_k$ be a factorization of $w$ into irreducibles. Since $w$ has a unique factorization of the form $w=xy$ where $x$ is irreducible in $M$ and $y\in M$, we must have $x_1=x$ and $x_2\cdots x_k=y$. By the induction hypothesis, $y$ has a unique factorization into irreducibles, so $x_2,\dots, x_k$ are uniquely determined.
\end{proof}

Let us say that a submonoid $M$ of the free monoid $A^*$ satisfies \emph{Sch\"utzenberger's criterion} if it has the property that for every $p$, $q$, and $r$ in $A^*$, if $p$, $pq$, $qr$, and $r$ are in $M$ then $q$ is in $M$. The following useful result is due to Sch\"utzenberger \cite[Theorem 1.4]{schutz}; see also Tilson \cite{tilson}.

\begin{lem}
\label{l-schutz}
Let $M$ be a submonoid of the free monoid $A^*$. Then $M$ is free if and only if $M$ satisfies Sch\"utzenberger's criterion.
\end{lem}
\begin{proof}
First, suppose that $M$ satisfies Sch\"utzenberger's criterion. It is enough to show that the hypothesis of Lemma \ref{l-prefix} holds. Let $w$ be a nonempty word of $M$, and suppose that $w=xv$ where $x$ is irreducible in $M$ and $v\in M$. Suppose also that $w$ can also be factored as $xy\cdot z$ where $y$ is nonempty and $xy$ and $z$ are in $M$. It is enough to show that $xy$ is not irreducible. Since $x$, $yz=v$, $xy$, and $z$ are in $M$, by Sch\"utzenberger's criterion we have $y\in M$, and this implies that $xy$ is not irreducible.

Next, suppose that $M$ is free, and suppose that $p$,  $pq$, $qr$, and $r$ are in $M$. Let $w=pqr$. Then $w$ has a unique factorization
$w=u_1u_2\cdots u_k$ into primes of $M$. The factorizations $w=p\cdot qr=pq\cdot r$ into elements of $M$ imply that for some $i\le j$, we have
$p=u_1\cdots u_i$, $q=u_{i+1}\cdots u_j$, and $r=u_{j+1}\cdots u_k$, and thus $q\in M$, so  Sch\"utzenberger's criterion holds.
\end{proof}

Now let $u$ and $v$ be words. We say that $u$ \emph{overlaps} with $v$ if there exist words $x$ and $y$ such that $ux=yv$ and $l(y) < l(u)$ (and thus $l(x)<l(v)$). For example, $ab$ overlaps with $bc$ because $ab\cdot c = a\cdot bc$. We call a word $w$ \emph{non-overlapping} if it does not overlap with itself. Thus $ab$ is non-overlapping, but $aa$ overlaps with itself.

Let $w$ be a word in the free monoid $A^*$. Let us denote by $A_w$ the set of words in $A^*$ that start with $w$, together with the empty word. Then $A_w$ is a submonoid of $A^*$.
If $A=\{a\}$ and $w=a^2$, then $A_w$ is not free, since $A_w$ is the set of words in $\{a\}^{\ast}$ of length not equal to 1. On the other hand, if $A=\{a,b\}$ and $w=ab$ then $A_w$ is easily seen to be free.

It will be conveniently to refer to $A_w$ as  ``the monoid of words in $A^*$ that start with $w$," and more generally, whenever we speak of the monoid of words with some property, it will be understood that the empty word is included, even if it does not have the property.

\begin{lem}\label{l-nonover}
The submonoid $A_w$ of the free monoid $A^{\ast}$ is free if and only if $w$ is non-overlapping.
\end{lem}\begin{proof} 

We first show that the condition is sufficient. Suppose that $w$ is non-overlapping. We will show that Sch\"utzenberger's criterion holds. Suppose that $p$, $pq$, $qr$, and $r$ are in $A_w$ and that $r$ is nonempty. Then since $qr$ and $r$ both start with $w$, and $w$ is non-overlapping, we must have $l(q)\ge l(w)$. This implies that since $qr$ starts with $w$, so does $q$, so $q\in A_w$.

For necessity, we show that if $w$ is overlapping then $A_w$ is not free. Suppose that $w$ is overlapping, so there exist words $t$, $u$, and $v$ such that
$t=wu=vw$, where $v$ (and thus also $u$) is shorter than $w$. We will show that $wt$ has two different factorizations into irreducibles in $A_w$. Any word in $A_w$ that is not irreducible must have length at least twice the length of $w$, so $wu$ and $wv$ are irreducible. Therefore $w\cdot wu$ and  $wv\cdot w$ are two different factorizations of $wt$ into irreducibles of $A_w$, so $A_w$ is not free.
\end{proof}

We can show similarly that if $u$ does not overlap with $v$ then the submonoid of $A^*$ of words that start with $u$ and end with $v$ is free.

\section{Fibonacci Compositions}

Let us define a {\it Fibonacci composition} of $n$ to be a composition of $n$ with parts 1 and 2. Thus the set of all Fibonacci compositions is the free monoid $\{1,2\}^*$. Most of our results are consequences of the fact there are $F_{n+1}$ Fibonacci compositions of $n$. (Many other identities are proved using this interpretation of Fibonacci numbers in Benjamin and Quinn \cite{bq}.)

Applying \eqref{e-freemonoid} to this free monoid with the weight function $\w(1)=1$, $\w(2)=2$, together with the fact that there are $F_{n+1}$ Fibonacci compositions of $n$, gives the generating function
\begin{equation*}
\sum_{n=0}^\infty F_{n+1}x^n =\frac{1}{1-x-x^2},
\end{equation*}
which is equivalent to \eqref{e-Fn}. Our proofs of $(i)$--$(vi)$ and other similar formulas are all based on free submonoids of the free monoid of Fibonacci compositions.

Now let us consider the monoid  $\{1,2\}_1$ of Fibonacci compositions that start with 1 (including the empty Fibonacci composition). By Lemma
\ref{l-nonover}, this monoid is free (though this is easy to see directly), and  the primes are the compositions of the form $\comp(1,2^i)$, for $i\ge0$.  So the generating function for primes in this free monoid is $\sum_{i=0}^\infty x^{2i+1}$.

It follows that if $s_n$ is the number of Fibonacci compositions of $n$ that start with 1, with $s_0=1$, then 
\begin{equation}
\label{e-odd}
\sum_{n=0}^\infty s_n x^n = \binv{-\sum_{i=0}^\infty x^{2i+1}},
\end{equation}
and the right side of \eqref{e-odd} is the generating function for compositions with odd parts. But the number of Fibonacci compositions of $n$ that start with 1 is just the number of Fibonacci compositions of $n-1$, so we see that for $n>0$, the number of compositions of $n$ with odd parts is equal to the number of Fibonacci compositions of $n-1$, which is $F_n$. So $(iii)$ holds. (This result seems to have been first given by Hoggatt \cite{hoggatt} and  Hoggatt and Lind \cite{hl3}.)

Our approach gives a simple bijective proof of this fact. Suppose that $c$ is a Fibonacci composition of $n-1$. Thus $\comp(1,c)$ is a Fibonacci composition of $n$ that starts with 1, which can be expressed uniquely as $\comp(1, 2^{i_1},1,2^{i_2}, \cdots, 1,2^{i_k})$. Then the corresponding composition of $n$ with odd parts is $\comp(1+2i_1, 1+2i_2,\cdots, 1+2i_k)$.

We could apply exactly the same analysis with the roles of 1 and 2 switched, and we would find that the number of Fibonacci compositions of $n$ that start with  2 is equal to the number of compositions of $n$ with all parts greater than 1, and this is $(ii)$.

A similar result applies to compositions with any set of two parts (cf.~Zeilberger \cite{zeilberger}, Sills \cite{sills}, and Munagi \cite[Theorem 1.2]{munagi}):

\begin{prop}
\label{p-2parts}
Let $p$ and $q$ be distinct integers. Then for $n\ge p$, the number of compositions of $n-p$ with parts $p$ and $q$ is equal to the number of compositions of $n$ with parts of the form $p+qi$, where $i\in \N$.
\end{prop}
\begin{proof}
First we note that prepending a part $p$ to a composition of $n-p$ with parts $p$ and $q$ gives a composition of $n$ that starts with $p$. In the free monoid $\{p,q\}^*_p$ of compositions with parts $p$ and $q$ that start with $p$, the primes are compositions of the form $\comp(p, q^i)$. Thus a bijection from the compositions of $n>0$ that start with $p$ to  the compositions of $n$ with parts of the form $p+qi$ is given by the map that takes $\comp(p, q^{i_1}, p, q^{i_2}, \cdots, p, q^{i_k})$ to the composition $\comp(p+q i_1, \,p+q i_2,\cdots, p+qi_k)$.
\end{proof}

The generating function identity that corresponds to Proposition \ref{p-2parts} is 
\begin{equation*}
1+\frac{x^p}{1-x^p-x^q}=\inv{-\frac{x^p}{1-x^q}}.
\end{equation*}

More generally, we can show that for any $k\ge1$, 
\[1+x^k\sum_{n=0}^\infty F_{n+1}x^n = 1+\frac{x^k}{1-x-x^2}\]
is the generating function for a free monoid.

\begin{prop}
\label{p-fmgen} Fix an integer $k\ge 2$. The monoid $M$ of Fibonacci compositions starting with $\comp(2, 1^{k-2})$ is a free monoid in which the primes are  of the form $\comp(2, 1^{k-2}, 1^i, q)$, where  $i$ is a nonnegative integer and $q$ is empty or  is a Fibonacci composition that starts with $2$ and contains no  $\comp(2, 1^{k-2})$. 

The generating function for the primes of $M$ is 
$
 {x^k}/({1-x-x^2+x^k})
$, and thus we have a combinatorial interpretation to the identity:
\begin{equation}
1+\frac{x^k}{1-x-x^2}=\biggl( 1- \frac{x^k}{1-x-x^2+x^k}
\biggr)^{-1}. \label{e-fmgen}
\end{equation}
\end{prop}

\begin{proof}
By Lemma \ref{l-nonover}, $M$ is a free monoid, and it is easy to see that the primes of $M$ are as stated in the proposition.
Let $Q$ be the monoid of Fibonacci compositions that start with  $2$ and contain no  
$\comp(2, 1^{k-2})$. Then $Q$ is a free monoid in which the set of primes consists of compositions $\comp(2, 1^{i})$,  with $0\le i\le k-3$. Thus the generating function for $Q$ is 
\[\biggl(1 -\sum_{j=2}^{k-1} x^j\biggr)^{-1}=\frac{1-x}{1-x-x^2+x^k},\]
and the generating function for the primes of $M$ is 
\begin{equation*}
\frac{x^k}{1-x}\cdot \frac{1-x}{1-x-x^2+x^k}=\frac{x^k}{1-x-x^2+x^k}.
\qedhere
\end{equation*}

\end{proof}

For $k=2$, $M$ is the free monoid of compositions that start with 2, and the primes, as we saw before, are of the form $\comp(2, 1^i)$, for $i\ge0$. 

For $k=3$, the generating function for the primes of $M$ 
is 
\begin{equation*}
\frac{x^3}{1-x-x^2+x^3}=\frac{x^3}{(1-x)^2(1+x)}=\frac{x^3+x^4}{(1-x^2)^2}
  =\sum_{n=3}^\infty \left\lfloor \frac{n-1}{2}\right\rfloor x^n
\end{equation*}
This  gives the formula 
\begin{equation}
\label{e-floor}
F_{n-2}=\sum_{a\in C(n)} \floor{\frac{a_1-1}{2}}\cdots \floor{\frac{a_k-1}{2}},
\end{equation}
for $n\ge2$. 
(Recall that $C(n)$ is the set of compositions $\comp(a_1,\dots, a_k)$ of $n$.)
We can explain this formula combinatorially by showing that there are $\floor{(n-1)/2}$ primes of $M$ of weight $n$. In this case, the primes are of the form $\comp(2,1^{i+1},2^j)$ for $i,j\in \N$. If such a word is a composition of $n$, then $i+1+2(j+1)=n$ so $i=n-2j-3$, and this is nonnegative for $j=0,1,\dots, \floor{n/2}-1$.

There is a slightly simpler interpretation of \eqref{e-floor}. Instead of  Fibonacci compositions that start with $\comp(2,1)$, we consider  Fibonacci compositions that start with 1 and end with 2. These form a free monoid in which the primes are of the form $\comp(1^i, 2^j)$, where $i,j\ge1$,  and there are $\lfloor (n-1)/2\rfloor$ of them of weight $n$. 

The cases when $k=2,3, 4,$ or $5$ in \eqref{e-fmgen} are specializations of the identities
\begin{align}
1+\frac{a}{1-a-b} &= \biggl( 1-\frac{a}{1-b}\biggr)^{-1},\label{e-k2}\\
1+\frac{ab}{1-a-b} &= \biggl( 1-\frac{ab}{(1-a)(1-b)}\biggr)^{-1},\label{e-k3}\\
1+\frac{a^2b}{1-a-b}&= \biggl(1-\frac{a^2b}{(1-a)(1-b-ab)}\biggr)^{-1},\label{e-k4}
\end{align}
which can also be interpreted in terms of free monoids.
 
There are other interesting applications of these formulas.  Taking $a=b=x$ in \eqref{e-k2}, shows that the total number of compositions of $n$, for $n>0$, is $2^{n-1}$. Taking $a=b=x$ in \eqref{e-k3} gives
 \begin{equation*}
2^{n-2}=\sum_{a\in C(n)}(a_1-1)\cdots (a_k-1), \quad n\ge 2.
\end{equation*}
Taking $a=b=x$ in \eqref{e-k4}, and using the fact that 
\begin{equation*}
\frac{x^3}{(1-x)(1-x-x^2)} = \frac{x}{1-x-x^2}-\frac{x}{1-x}=\sum_{n=1}^\infty (F_n -1) x^n,
\end{equation*}
gives
\begin{equation}
\label{e-Fsum}
2^{n-3}=\sum_{a\in C(n)} (F_{a_1}-1)\cdots (F_{a_k}-1), \quad n\ge 3.
\end{equation}
(the nonzero terms come from compositions into parts greater than 2). Formula \eqref{e-Fsum} ``explains" why the first three nonzero values of $F_n-1$ are the first three powers of 2 (i.e., $F_3-1=1$, $F_4-1=2$, $F_5 -1 = 4$), since for $3\le n<6$ there is just one nonzero term in the sum in \eqref{e-Fsum}. That $F_n-1=2^{n-3}$ for $3\le n < 6$  can also be seen from the formula
\begin{equation*}
\frac{x}{1-x-x^2}-\frac{x}{1-x}=\frac{x^3}{1-2x+x^3}.
\end{equation*}

Taking $a=x$, $b=x^m$ in \eqref{e-k3} gives a generalization of \eqref{e-floor}:

\begin{prop}
\label{p-mfib}
Fix $m\ge1$ and define the numbers $r_n$ by 
\begin{equation*}
\sum_{n=0}^\infty r_n x^n = \frac{1}{1-x-x^m}.
\end{equation*}
Then for $n\ge m+1$ we have
\begin{equation*}
r_{n-m-1} = \sum_{a\in C(n)} \floor{\frac{a_1-1}{m}}\cdots \floor{\frac{a_k-1}{m}},
\end{equation*}
where the only nonzero terms come from compositions in which every part is at least $m+1$.
\end{prop}

\begin{proof}
By \eqref{e-k3} with $a=x$, $b=x^m$, we have
\begin{equation*}
1+\sum_{n=m+1}^{\infty} r_{n-m-1}x^n= 1+\frac{x^{m+1}}{1-x-x^m}=
\left(1-\frac{x^{m+1}}{(1-x)(1-x^m)}\right)^{-1}.
\end{equation*}
We have
\begin{align*}
\frac{x^{m+1}}{(1-x)(1-x^m)} 
  &=x^{m+1}\frac{1+x+\cdots+x^{m-1}}{(1-x^m)^2}\\
  &=x (x^m+x^{m+1}+\cdots+x^{2m-1})(1+2x^m + 3x^{2m}+\cdots)\\
  &=\sum_{n=1}^\infty \floor{\frac{n-1}m}x^n
\end{align*}
and the result follows.
\end{proof}

It is not hard to give a combinatorial interpretation to 
Proposition \ref{p-mfib}. The number $r_n$ counts compositions of $n$ with parts 1 and $m$, so for $n\ge m+1$, $r_{n-m-1}$ is the number of such compositions of $n$ that start with 1 and end with $m$. These compositions form a free monoid in which the primes are of the form 
form $\comp(1^i, m^j)$, where $i,j\ge1$,  and there are $\lfloor (n-1)/m\rfloor$ of them of weight~$n$.

The numbers $r_n$ for $m=3,4,5,6,\dots,15$ are sequences \seqnum{A000930}, \seqnum{A003269}, \seqnum{A003520},   \seqnum{A005708}--\seqnum{A005711}, and 
\seqnum{A017898}--\seqnum{A017909}.

\section{Multisection}
\label{s-multi}

To explain results such as $(iv)$--$(vi)$ of Section \ref{s-intro}, we need to consider Fibonacci compositions of only even or only odd numbers. More generally, given $m$ and $i$, we may consider Fibonacci compositions of numbers congruent to $i$ modulo $m$.

The following result, which follows easily from Lemma \ref{l-schutz}, tells us that the relevant monoids are free.

\begin{lem}
\label{l-multisect}
Let $M$ be a free monoid with a weight function and let $m$ be a positive integer.  Then the submonoid of $M$ consisting of 
elements of weight divisible by $m$ is free.
\end{lem}

We shall apply Lemma \ref{l-multisect} to free monoids of Fibonacci words.
Let
\[f_{m,i}=\sum_{n=0}^\infty F_{mn+i+1} x^n,\]
so that the coefficient of $x^n$ in $f_{m,i}$ is the number of Fibonacci compositions of $mn+i$. 

It follows from Lemma \ref{l-multisect} that the monoid of Fibonacci compositions of multiples of $m$ is free. Let us define a weight function on this monoid by taking the weight of a composition of $mn$ to be $n$. Then the generating function for this monoid is $f_{m,0}$. 

Now let $w$ be a non-overlapping Fibonacci composition of an integer $r=mk-i$, where $k\ge1$ and $0\le i <m$. Then by Lemmas \ref{l-nonover} and \ref{l-multisect}, the monoid of Fibonacci compositions  of multiples of $m$,  starting with $w$, is a free monoid, and the generating function for this free monoid is $1+x^k f_{m,i}$.

In this section, we consider a few cases of these generating functions for arbitrary $m$, and in sections \ref{s-bisection} and \ref{s-trisection} we look in more detail at the cases $m=2$ and $m=3$.

It is not hard to show (e.g., by using the Binet formula for Fibonacci numbers) that
\begin{equation}
\label{e-multi}
\sum_{n=0}^\infty F_{mn+j} x^n=\frac{F_j + (-1)^j F_{m-j} x}{  1-L_m x+(-1)^m x^2},
\end{equation}
where $L_n$ is the $n$th Lucas number ($L_0=2$ and $L_n = F_{n-1}+F_{n+1}$). See, for example, Hoggatt and Lind \cite[equation (4.18)]{hl}.
We can find simple free monoid interpretations for  generating functions for $F_{mn+1}$ and $F_{mn-1}$.

For $F_{mn+1}$, we have the following formula, due to Hoggatt \cite{hoggatt}.

\begin{prop}
\label{p-m0}
\begin{equation}
\label{e-m0}
f_{m,0} = \sum_{n=0}^\infty F_{mn+1}x^n 
=\frac{1-F_{m-1}x}{1-L_m x + (-1)^m x^2}
=
\left( 1- F_{m+1}x - \frac{F_m^2 x^2}{1-F_{m-1}x}\right)^{-1}.
\end{equation}
\end{prop}

\begin{proof} The formula is a straightforward computation, using the case $j=1$ of \eqref{e-multi}, the formula $L_m=F_{m-1}+F_{m+1}$, and Cassini's identity
$F_{m+1}F_{m-1} - (-1)^m = F_m^2$.
\end{proof}
We can interpret Proposition \ref{p-m0} combinatorially by describing the primes of the free monoid of Fibonacci compositions of multiples of $m$. The primes of weight one are the Fibonacci compositions of $m$, which are counted by $F_{m+1}$. The primes of weight $n>1$ are of the form 
$\comp(u, 2, v_1, 2, v_2, 2,\dots, 2,v_k, 2, w)$ where $u$ and $w$ are Fibonacci compositions of $m-1$ and each $v_i$ is a Fibonacci composition of $m-2$.

There is an analogous formula for $F_{mn-1}$, corresponding to compositions of multiples of $m$ that start with 2, which have the generating function $1+xf_{m,m-2}$. A straightforward  computation gives the following result, also due to Hoggatt \cite{hoggatt}.

\begin{prop}
\label{p-m m-2}
\begin{equation}
\label{e-mm-2}
1+xf_{m,m-2} = 1+\sum_{n=1}^\infty F_{mn-1}x^n =\left(1- F_{m-1}x - \frac{F_m^2 x^2}{1-F_{m+1}x}\right)^{-1}.
\end{equation}
\end{prop}

There is a simple combinatorial interpretation for Proposition \ref{p-m m-2}. In the free monoid of Fibonacci compositions of multiples of $m$ that start with 2, the primes of weight 1 are the Fibonacci compositions of $m$ that start with 2, and there are $F_{m-1}$ of them. Every other prime of this free monoid is of the form $\comp(2, u, v_1, \dots, v_k, w)$, for $k\ge0$, where $u$ and $w$ are Fibonacci composition of $m-1$ (counted by $F_m$) and each $v_i$ is a Fibonacci composition of $m$ (counted by $F_{m+1}$).

There don't seem to be results as simple as Propositions \ref{p-m0} and \ref{p-m m-2} for $f_{m,i}$ with $i$ not equal to~0 or $m-2$.

For $i=m-1$, we have by \eqref{e-multi} 
\begin{equation}
\label{e-mm-1}
f_{m,m-1}=\sum_{n=0}^\infty F_{m(n+1)}=\frac{F_m}{1-L_mx+(-1)^mx^2},
\end{equation}
and a straightforward calculation gives
\begin{equation}
\label{e-mm-0}
1+xf_{m,m-1}
=\biggl(1-\frac{F_m x}{1-2F_{m-1}x+(-1)^m x^2}\biggr)^{-1}.
\end{equation}
It is possible to describe the corresponding primes explicitly, but there doesn't seem to be  a simple combinatorial explanation for their generating function  \eqref{e-mm-0}.

We note also that if $m$ is odd then by \eqref{e-multi},
\[f_{m,m-1}=\sum_{n=0}^\infty F_{mn+m} x^n = \frac{F_m}{1-L_m x -x^2}.\]
Although this formula looks like it should have a simple combinatorial explanation, there does not seem to be one.

In the next two sections we consider bisection and trisection in more detail.

\section{Bisection}
\label{s-bisection}
We now consider the case $m=2$ of \eqref{e-multi}, which gives 
\begin{equation*}
f_{2,0}=\frac{1-x}{1-3x+x^2}
\end{equation*}
and
\begin{equation*}
f_{2,1}=\frac{1}{1-3x+x^2}.
\end{equation*}
We leave it to the reader to give a combinatorial interpretation of the formula $f_{2,1} = (1-x)^{-1}f_{2,0}$.

The case $m=2$ of \eqref{e-m0} is 
\begin{equation}
\label{e-bisfibo0}
    f_{2,0}=\frac{1-x}{1-3x+x^2} =(1-2x-x^2-x^3-\cdots)^{-1} 
    =\biggl( 1- 2x - \frac{x^2}{1-x}
    \biggr)^{-1}. 
\end{equation}
and the primes of the free monoid of Fibonacci compositions of even integers are the composition~$\comp(2)$ and compositions of the form $\comp(1, 2^i, 1)$ for non-negative integers $i$. The generating function for the primes, together with~\eqref{e-freemonoid}, gives  formula $(vi)$ of section \ref{s-intro},
\begin{equation*}F_{2n+1} =\mathsum_{a\in C(n)} 2^{\#\{i: a_i=1\}}.\end{equation*}

We also have the identity
\begin{align}
1+x^k f_{2,0} &=  1+\frac{x^k(1-x)}{1-3x+x^2}  =\biggl( 1 - \frac{x^k(1-x)}
   {1-3x+x^2+x^k-x^{k+1}} \biggr)^{-1}, 
    \label{e-fibogenx}
\end{align}
so ${x^k(1-x)}/(
   {1-3x+x^2+x^k-x^{k+1}})$
is the generating function for primes in the free monoid of Fibonacci compositions of even integers that start with $\comp(2,1^{2k-2})$. Only the case $k=1$ of \eqref{e-fibogenx}, which is also the case $m=2$ of \eqref{e-mm-2}, is especially simple (though the prime counting sequence for $k=2$ is \seqnum{A052921}). Here we have
\begin{equation*}
 1+x f_{2,0} = \left( 1- x-\frac{x^2}{1-2x}\right)^{-1}=(1-x-x^2-2x^3-4x^4 - 8x^5-\cdots)^{-1},
\end{equation*}
which gives the formula
\[F_{2n-1} = \sum_{a\in C(n)} 2^{\#\{i: a_i=1\}+n-2k},\]
where $k$ is the number of parts of the composition $a$.

We note also that the continued fraction formula
\begin{equation*}
1+xf_{2,0} = \frac{1-2x}{1-3x+x^2} = 
 \cfrac{1}{1-\cfrac{x}{1-\cfrac{x}{1-x}}}
\end{equation*}
shows that $F_{2n-1}$, for $n>0$, is the number of Dyck paths of length $2n$ and height at most 3 (see, e.g., Flajolet \cite{flajolet}).

The analogous formulas for $f_{2,1}$ are somewhat simpler than those for $f_{2,0}$.

\begin{prop}
\label{p-sgen}
Fix an integer $k\ge 1$. The set of Fibonacci compositions of even numbers that start with $\comp(1, 2^{k-1})$ is a free monoid whose primes are of the form $\comp(1, 2^{k-1}, 2^i, q, 1, 2^j)$, where $i$ and $j$ are nonnegative integers and  $q$ is an element of the free monoid $Q$  with primes \[\comp(1, 2^l, 1, 2^m), \text{ for }l\ge 0 \text{ and } 0\le m\le k-2.\] 

This gives identity
\begin{equation}
1+x^k f_{2,1}=1+\frac{x^k}{1-3x+x^2}=\biggl( 1-\frac{x^k}{1-3x+x^2+x^k}\biggr)^{-1}.
\label{e-sgen}
\end{equation}
For $k=1$, the free monoid $Q$ contains only the empty word, and the primes are of the form
$\comp(1, 2^i,1,2^j)$ for $i,j\ge0$.
\end{prop}

\begin{proof}
Let $M$ be the monoid of Fibonacci compositions of even numbers that start with $\comp(1, 2^{k-1})$. By Lemma~\ref{l-nonover},  $M$ is a free monoid. A prime of $M$ is a Fibonacci composition that starts with $\comp(1, 2^{k-1})$ and contains an even number of 1s, where the $j$th part 1, for $j>1$ and odd, is followed by at most $k-2$ parts 2. It is clear that these primes are as described in the proposition.

The generating function for $Q$ is
  \[u(x)=\biggl( 1- \sum_{i=1}^{k-1}\frac{x^i}{1-x}\biggr)^{-1}
  =\frac{(1-x)^2}{1-3x+x^2+x^k}.\]
Thus the generating function for the primes of $M$ is
  \begin{align*}
   & \phantom{=} x^k \cdot \frac{1}{1-x}\cdot  u(x) \cdot \frac{1}{1-x}= \frac{x^k}{1-3x+x^2+x^k}.\qedhere
  \end{align*}
\end{proof}

For $k=1$, \eqref{e-sgen} gives
\[1+\frac{x}{1-3x+x^2}=\biggl( 1-\frac{x}{(1-x)^2}\biggr)^{-1}=(1-x-2x^2-3x^3-\cdots)^{-1}.\]
We can see directly that there are $n$ prime compositions of $2n$, since these are compositions of $2n$ of the form $\comp(1,2^i,1,2^j)$ and there are $n$ solutions of $i+j=n-1$. 

This identity gives  formula $(iv)$ of Section \ref{s-intro},
\begin{equation}
\label{e-iv}
F_{2n} = \mathsum_{a\in C(n)} a_1 a_2 \cdots a_k,
\end{equation}
for $n\ge 1$, as shown by Moser and Whitney \cite{mw}.

Stanley \cite[p.~52]{ec1} gives a combinatorial interpretation of \eqref{e-iv}: The sum $\mathsum a_1 a_2 \cdots a_k$ is the number of ways of inserting at most one vertical bar in each of the $n-1$ spaces separating a line of $n$ dots, and then circling one dot in each compartment. Replacing each bar by a 1, each uncircled dot by a 2, and each circled dot by a 1 gives all Fibonacci compositions of $2n-1$ exactly once. As an example for $n=8$, we have
\[ \bullet \,\ \odot \ \mid \ \odot\, \ \bullet \ \mid \ \odot \ \mid \bullet \ \bullet \ \odot \quad \Longleftrightarrow \quad 21112111221
\]
We can explain this bijection in terms of our free monoid approach: if we insert a bar at the beginning of one of these arrangements of bars and dots, then we have a sequence of configurations of the form
$\mid\ \bullet^i\ \odot\ \bullet^{\kern .08 em j}$, and this configuration corresponds to the prime $\comp(1,2^i,1,2^j)$.

For $k=2$, the generating function for the primes in Proposition~\ref{p-sgen} is
\[\frac{x^2}{1-3x+2x^2}=\frac{x^2}{(1-2x)(1-x)}=\mathsum_{n=2}^\infty \mathsum_{i=0}^{n-2} 2^i x^n = \mathsum_{n=2}^\infty (2^{n-1}-1) x^n .\]

This gives formula $(v)$ of Section \ref{s-intro}:
\begin{equation}
\label{e-v}
F_{2n-2}=\mathsum_{a\in C(n)} (2^{a_1-1}-1)(2^{a_2-1}-1)\cdots (2^{a_k-1}-1),
\end{equation}
for $n\ge 1$.
From Proposition \ref{p-sgen}, the prime compositions corresponding to \eqref{e-v} are of the form 
\[\comp(1,2, 2^{l_0}, 1, 2^{l_1}, 1^2, 2^{l_2}, 1^2,\cdots, 1^2, 2^{l_m}, 1^2, 2^{l_{m+1}}),\]
with $m\ge0$,  where
$l_0,\dots, l_{m+1}$ are nonnegative integers. 
To see combinatorially that there are $2^{n-1}-1$ such compositions of $2n$, we start 
with the composition of $2n-2$  with $n-1$ parts, all equal to 2. We choose some nonempty subset of the parts, which we can do in $2^{n-1}-1$ ways. We replace the first selected 2 with 1 and replace each other selected 2 with $\comp(1,1)$. Finally we insert $\comp(1,2)$ at the beginning.

For $k>2$, the primes in Proposition \ref{p-sgen} are more complicated, but for $k=3$ and $k=4$ the  sequences that count them are \seqnum{A048739} and \seqnum{A077849}.

\section{Trisection}
\label{s-trisection}

Now we consider trisections of the Fibonacci sequence, for which we have
\begin{gather*}
f_{3,0}=\frac{1-x}{1-4x-x^2}\\
f_{3,1} = \frac{1+x}{1-4x-x^2}\\
f_{3,2}= \frac{2}{1-4x-x^2}.
\end{gather*}

Proposition \ref{p-m0} and its combinatorial interpretation give us a combinatorial interpretation to the identity
\begin{equation}
\label{e-trisec0} 
f_{3,0} 
=(1-3x-4x^2-4x^3-\cdots)^{-1}=\biggl(1-3x-4\,\frac{x^2}{1-x}\biggr)^{-1},
\end{equation}
which gives the formula
\begin{align*}
F_{3n+1} = \mathsum_{a\in C(n)} 3^{\#\{i\,:\, a_i=1\}} 4^{\#\{j\,:\, a_j\neq 1\}},
\end{align*}
for $n\ge 0$:
the free monoid of Fibonacci compositions of numbers divisible by 3  has three primes of weight one,  $\comp(1,1,1)$, 
$\comp(1,2)$,  and $\comp(2,1)$, and four primes of weight $n$ for each $n\ge 2$, each of which is of the form
\[\comp(a, b, c),\] where $a$ and $c$ are either $\comp(1,1)$ or $\comp(2)$, and $b$ is the composition
\[\comp(2,1,2,1,\cdots, 2,1,2)\] with $n-2$ parts 1 and $n-1$ parts 2. 

Next, we give a free monoid interpretation to the case $m=3$ of  \eqref{e-mm-0}:
\begin{equation}
1+xf_{3,2}=1+\frac{2x}{1-4x-x^2}=\biggl(1-\frac{2x}{1-2x-x^2}\biggr)^{-1}.
\label{e-32x}
\end{equation}

The left side of~\eqref{e-32x} counts Fibonacci compositions of numbers divisible by $3$ that start with a part $1$. In this free monoid there are two primes of weight one, $\comp(1,1,1)$ and $\comp(1,2)$. For  $n\ge 2$, a prime $p$ of weight $n$ ends with $\comp(2,1)$, $\comp(2,1,1)$ or $\comp(2,2)$, so it must be of one the following three kinds:
\begin{enumerate}[(1)]
\item $p$ is obtained from a prime of weight $n-1$ by attaching $\comp(2,1)$ in the end;
\item $p$ is obtained from a prime of weight $n-1$ ending with a part 1 by replacing the 1 with $\comp(2,1,1)$;
\item $p$ is obtained from a prime of weight $n-1$ ending with a part 1 by replacing  the 1 with $\comp(2,2)$.
\end{enumerate}
Let $a_n$ be the number of primes of weight $n$ and let $b_n$ be the number of primes of weight $n$ that end with a part 1. Then we have $a_n= a_{n-1} + 2 b_{n-1}$ and $b_n = a_{n-1} + b_{n-1}$.  
Thus 
\begin{align*}
a_n &= a_{n-1}+2b_{n-1} =a_{n-1}+2(a_{n-2}+b_{n-2})\\
  &=(a_{n-1}+a_{n-2}) + (a_{n-2}+2b_{n-2}) = (a_{n-1}+a_{n-2}) +a_{n-1}\\
  &=2a_{n-1}+a_{n-2}
\end{align*}
for $n\ge3$, and since $a_2 = a_1+2b_1 = 2 + 2 =4=2a_1+a_0$, where $a_0=0$, the recurrence holds for $n\ge 2$. Thus we find that the generating function for primes is 
\begin{equation}
\label{e-pell}
\frac{2x}{1-2x-x^2},
\end{equation}
and~\eqref{e-32x} follows.
The coefficients of \eqref{e-pell} are sequence \seqnum{A052542} or \seqnum{A163271}. They are twice the Pell numbers, sequence \seqnum{A000129}.

The case $m=3$ of Proposition \ref{p-m m-2} is
\begin{equation*}
1+xf_{3,1} = \left(1-x-\frac{4x^2}{1-3x}\right)^{-1}
= \left(1- x - \sum_{n=1}^\infty 4\cdot 3^{n-2}x^n\right)^{-1},
\end{equation*}
where the counting sequence for the primes is \seqnum{A003946}

We note also that
\begin{equation*}
1+xf_{3,0}=\left(1-\frac{x(1-x)}{1-3x-2x^2}\right)^{-1},
\end{equation*}
where the primes are counted by  \seqnum{A104934}, and
\begin{equation*}
1+xf_{4,3} = \left( 1-\frac{3x}{1-4x+x^2}\right)^{-1},
\end{equation*}
where the primes are counted by \seqnum{A005320}.

\end{document}